\documentclass[final,12pt]{article}
\usepackage{mathtools}
\usepackage{amsmath}
\usepackage{amsfonts}
\usepackage{amssymb}
\usepackage{amsthm}
\usepackage{graphicx}
\usepackage{mathrsfs}
\usepackage{xcolor}

\usepackage{enumerate}
\numberwithin{equation}{section}
\usepackage[margin=2cm]{geometry}
\newtheorem{theorem} {Theorem}[section]
\newtheorem{proposition} [theorem]{Proposition}
\newtheorem{lemma}[theorem]{Lemma}

\newtheorem{definition}[theorem] {Definition}

\newcommand{\pa} {\partial}
 \newcommand{\al} {\alpha}
\newcommand{\rr}{\rightarrow}
\newcommand{\B} {\beta}
\newcommand{\de} {\delta}

\newcommand{\ga}{\gamma}
\newcommand{\p}  {\prime}
\newcommand{\e}  {\epsilon}
\newcommand{\De} {\Delta}

\newcommand{\si} {\sigma}
\newcommand{\La} {\Lambda}
\newcommand{\f}{\infty}
\newcommand{\R}{\mathbb{R}}
\newcommand{\noi} {\noindent}

\DeclareMathOperator{\dv}{div}

\newcommand{\norm}[1]{\left|\hspace{-0.2mm}\left| #1 \right|\hspace{-0.2mm}\right|}
\newcommand{\abs}[1]{\left| #1\right|}

\date{}

\newcommand{\U}{\textbf{u}}
\newcommand{\V}{\textbf{v}}
\newcommand{\vr}{\varrho}
\newcommand{\nb}{\nabla}
\newcommand{\dom}{\int\limits_{\Omega}}
\newcommand{\domm}{\int\limits_{\tau_1}^{\tau_2}\int\limits_{\Omega}}

\title{Weak-Strong Uniqueness for the Isentropic Euler Equations with Possible Vacuum}

\author{Shyam Sundar Ghoshal\footnote{Centre for Applicable Mathematics, Tata Institute of Fundamental Research, Post Bag No 6503, Sharadanagar, Bangalore -- 560065, India. E-mail: ghoshal@tifrbng.res.in}, Animesh Jana\footnote{Centre for Applicable Mathematics, Tata Institute of Fundamental Research, Post Bag No 6503, Sharadanagar, Bangalore -- 560065, India. E-mail: animesh@tifrbng.res.in},  and Emil Wiedemann\footnote{Institute of Applied Analysis, Ulm University, Helmholtzstr.~18, 89081 Ulm, Germany. E-mail: emil.wiedemann@uni-ulm.de}}

\begin{document}
\maketitle

\begin{abstract}
We establish a weak-strong uniqueness result for the isentropic compressible Euler equations, that is: As long as a sufficiently regular solution exists, all energy-admissible weak solutions with the same initial data coincide with it. The main novelty in this contribution, compared to previous literature, is that we allow for possible vacuum in the strong solution. 
\end{abstract}

\section{Introduction}
We consider the Cauchy problem for the isentropic Euler system in arbitrary space dimension:
\begin{align}
\pa_t\vr +\dv_x(\vr\U)&=0,\label{isen1}\\
\pa_t(\vr\U) +\dv_x(\vr\U\otimes\U)+\nb_xp(\vr)&=0,\label{isen2}\\
(\vr,\U)(0,x)&=(\vr_0,\U_0),\nonumber
\end{align}
where $p(\varrho)=\kappa\varrho^\gamma$ and $\gamma>1$ is the \emph{adiabatic exponent} of the fluid. For reasons discussed below, we will consider solutions only on the whole space $\R^N$; our results remain valid on the torus, but presumably not on domains with physical boundaries. 

Since, through the mechanism of shock formation (or, conceivably, otherwise), an initially smooth solution may become discontinuous in finite time, it is standard to consider \emph{weak solutions} to hyperbolic conservation laws: 

\begin{definition}
We say $(\varrho,\textbf{u})\in C\left([0,T],L_{loc}^1(\R^N;\R_+\times\R^N)\right)$ is a weak solution to the isentropic Euler system \eqref{isen1}--\eqref{isen2} if also $\varrho\textbf{u}\in C\left([0,T],L_{loc}^1(\R^N;\R^N)\right)$ and if the following holds:

\begin{enumerate}
	\item 
	\begin{equation*}
	\int\limits_{\tau_1}^{\tau_2}\int\limits_{\R^N}\varrho\pa_t\varphi+\varrho\textbf{u}\cdot\nabla_x\varphi\,dxdt=\int\limits_{\R^N}\varrho\varphi(\tau_2,\cdot)\,dx-\int\limits_{\R^N}\varrho\varphi(\tau_1,\cdot)\,dx
	\end{equation*}
	for $\varphi\in C_c^{\f}([0,T]\times\R^N)$ and $0\leq\tau_1<\tau_2\leq T$;

    \item 
    \begin{equation*}
     \int\limits_{\tau_1}^{\tau_2}\int\limits_{\R^N}\varrho\textbf{u}\cdot\pa_t\psi+\varrho\textbf{u}\otimes\textbf{u}:\nabla_x\psi+p(\varrho)\dv_x\psi\,dxdt=\int\limits_{\R^N}\varrho\textbf{u}\cdot\psi(\tau_2,\cdot)\,dx-\int\limits_{\R^N}\varrho\textbf{u}\cdot\psi(\tau_1,\cdot)\,dx
    \end{equation*}
   for $\psi\in C_c^{\f}([0,T]\times\R^N,\R^N)$ and $0\leq\tau_1<\tau_2\leq T$.
\end{enumerate}
It is understood that all appearing integrands are integrable.
\end{definition}

It is well-known that conservation laws may exhibit `non-physical shocks' and thus nonuniqueness of the Cauchy problem. To rule out such unphysical solutions, one invokes the following energy admissibility condition:

\begin{definition} We say a weak solution of \eqref{isen1}--\eqref{isen2} is admissible if 
\begin{equation}\label{energy-inequality}
\int\limits_{\R^N}\left[\frac{1}{2}\varrho\abs{\textbf{v}}^2+H(\varrho)\right](\tau,\cdot)\,dx\leq\int\limits_{\R^N}\frac{1}{2}\varrho_0\abs{\textbf{v}_0}^2+H(\varrho_0)\,dx<\f
\end{equation}
for $0\leq\tau\leq T$, where $H(\varrho)=\frac{\kappa}{\gamma-1}\varrho^{\gamma}$ is the pressure potential.
\end{definition}

Note that this admissibility condition is much weaker than the usual entropy condition, which in the case of the isentropic Euler equations would take the form of a \emph{local} energy balance inequality, whereas~\eqref{energy-inequality} is an inequality only for the \emph{total} energy. Moreover, we do not require the total energy to be non-increasing in time, but merely ask that it do not exceed its initial value. It turns out, however, that this very weak admissibility property already suffices for weak-strong uniqueness.

The principle of weak-strong uniqueness states that each admissible weak solution with certain initial data coincides with the strong solution evolving from the same data, as long as the strong solution exists. Before we explain in detail what exactly a `strong solution' is supposed to be, let us emphasize that we cannot hope for an unconditional uniqueness statement. Indeed, admissible weak solutions of the isentropic Euler equations have been known to be non-unique for certain initial data, even under much stronger admissibility assumptioms than~\eqref{energy-inequality}. Constructions of such non-unique solutions rely on a so-called convex integration scheme~\cite{DLS2, Chiodaroli, Chio, AW, Vasseur}.

Weak-strong uniqueness for conservation laws was pioneered by C.~Dafermos and R.~DiPerna~\cite{Dafermos, DiPerna} and has since been applied and extended in various ways, see for instance~\cite{Chen1, BDLSz, DST, GGW, GKS}, Section 5.3 in the book~\cite{DafermosBook}, or the recent survey~\cite{Wiedemann}. While, in many articles, the strong solution is assumed $C^1$ or at least Lipschitz in space, with Lipschitz constant integrable in time, this regularity class does not quite include rarefaction waves, which are unique among admissible weak solutions nevertheless for the isentropic and full Euler systems~\cite{ChenChen, FK, FKV}. Motivated by the study of rarefaction waves, it has been observed that the strong solution does, generally, not need to be Lipschitz, but merely one-sided Lipschitz together with a certain degree of Besov regularity, in order to guarantee weak-strong uniqueness~\cite{FGJ, GJ, GJK}. 

However, all the mentioned results require the strong solution (although not the weak one) to have density bounded away from zero, so that vacuum states are excluded. In fact, the Euler system is strictly hyperbolic only on the set $\{\rho>0\}$, so that vacuum states always pose a significant obstacle to the analysis of compressible flow. In~\cite{ChenChen}, certain vacuum regions arising in rarefaction waves were treated; some Onsager-type energy conservation criteria were given for the isentropic Euler system with possible vacuum in~\cite{ADSW}; and a localized version of weak-strong uniqueness was given in~\cite{Wiedemann2} for the isentropic Euler equations with possible vacuum for adiabatic exponents $\gamma\geq 2$. For the compressible Navier-Stokes system, Feireisl and Novotn\'y~\cite{FN} very recently showed weak-strong uniqueness with possible vacuum in the strong solution on various types of domain.

Here, for the first time, we prove a weak-strong uniqueness theorem for strong solutions of the isentropic Euler equations whose density is not assumed bounded away from zero and which are not assumed to have any specific structure (like rarefaction waves). Instead, we show that conditions~\eqref{relation-theta-beta} and~\eqref{thetacondition} below suffice for weak-strong uniqueness. This condition is strictly weaker than boundedness away from zero and therefore admits certain vacuum states.

More precisely, we define a strong solution with possible vacuum as follows, inspired by \cite{LiuYang}:   

   \begin{definition}\label{def:regular-soln-1}
	We say $(\vr,\U)\in C([0,T), L^1(\R^N,[0,\f)\times\R^N))\cap L^\f([0,T)\times\R^N;[0,\f)\times\R^N)$ is a strong solution of the system \eqref{isen1}--\eqref{isen2} if $(\varrho,\U)$ is a weak solution to \eqref{isen1}--\eqref{isen2} and satisfies the following:
	\begin{enumerate}[(i)]
		\item\label{A1-1} $\U\in B^{\al,\f}_{2q}((\de,T)\times{\R^N})$, $\varrho\in B^{\B,\f}_{2q}((\de,T)\times{\R^N})$ with $\al\geq \B>\frac{1}{\min\{2,\ga\}}$ and $q\geq 2\ga/(\ga-1)$ for all $\de>0$. There exists a function $\eta\in C^1(\R^{N+1})$ with $supp(\eta)\subset\{\abs{(t,x)}\leq l\}$ for some $l>0$ and $\norm{\eta}_{L^1(\R^{N+1})}=1$, such that the following holds. Define $\mathcal{W}_\e[t_1,t_2]:=\{(x,t)\in [t_1,t_2]\times\R^N: \varrho_\e(x,t)>0\}$ for all $0<t_1<t_2$ where $\varrho_\e=\varrho*\eta_\e$ with $\eta_\e:=\e^{-(N+1)}\eta(z/\e)$. 
		The density $\varrho$ satisfies
		\begin{equation}\label{relation-theta-beta}
         \int\limits_{\mathcal{W}_\e[t_1,t_2]}\frac{1}{\varrho^{\theta}_\e(z)}\,dz\leq C_1
		\end{equation}
		where $C_1$ is independent of $\e$. {In the above inequality, the exponent $\theta$ satisfies
		\begin{equation}\label{thetacondition}
		\theta>\left\{\begin{array}{rl}
         \frac{4\ga^2}{(\ga-1)((\ga-1)^2\B+1-\B)}(1-\B)&\mbox{ if }1<\ga<2,  \\
         \frac{4\ga}{\ga-1}\frac{1-\B}{\B}&\mbox{ if }\ga\geq2.
		\end{array}\right.
		\end{equation}}
		
		%
		%
		%
		\item Moreover, in the exterior of the support of $\vr$, $\U$ is $C^1$ and it satisfies
		\begin{equation}\label{eqn:mom-vacuum-1}
		\pa_t\U+\U\cdot\nabla_x\U=0.
		\end{equation}
		\item\label{A1-3} There exists an integrable function $\La:[0,T]\rr[0,\f)$  such that the following one-sided Lipschitz condition holds:
		\begin{equation}\label{condition-1}
		\int_{\R^N}\left(-\xi\cdot\V(\cdot,t)\nb_x\varphi\cdot\xi+\La(t)|\xi|^2\varphi\right)\geq 0
		\end{equation}
		for any $\xi\in\R^N$, $t\in (0,T]$ and $\varphi\in C^{\f}_c(\R^N)$ with $\varphi\geq 0$.
	\end{enumerate}
\end{definition}


Here, $B^{\al,\infty}_q$ denotes a Besov space. There are several ways to define Besov spaces, but here the only property we use is 
\begin{equation*}
\|u_\e-u\|_{L^q(\Omega_\e)}\leq \epsilon^\alpha\|u\|_{B^{\al,\infty}_q(\Omega)},\quad\quad \|\nabla u_\e\|_{L^q(\Omega_\e)}\leq \epsilon^{\alpha-1}\|u\|_{B^{\al,\infty}_q(\Omega)} 
\end{equation*} 
for any $u\in B^{\al,\infty}_\infty(\Omega)$, where $u_\e$ denotes mollification and $\Omega_\e:=\{x\in\Omega: \operatorname{dist}(x,\partial\Omega)>\e\}$. See also~\cite{CET}.

Our main result can then be stated as follows:
\begin{theorem}\label{theorem:vacuum-1}
	Let $(\vr,\U)$ be an admissible weak solution of the system \eqref{isen1}--\eqref{isen2} with initial data $(\vr_0,\U_0)$. Let $(r,\V)$ be a regular solution of \eqref{isen1}--\eqref{isen2} according to Definition \ref{def:regular-soln-1} with same initial data $(\vr_0,\U_0)$. Then we have 
	\begin{equation*}
	\vr= r\mbox{ and }\vr\U=r\V\mbox{ in }[0,T]\times\R^N.
	\end{equation*}
\end{theorem}

Our proof relies on the well-known relative energy method and on suitable commutator estimates. However, commutators become tricky in presence of vacuum, so we need to estimate the crucial commutators very carefully. This is done in Lemma~\ref{lemma:commutator} below, which can be viewed as as major novel contribution of this work. In Section~\ref{mainproof} we show how to employ the commutator estimate within the relative energy framework and thus conclude the proof of Theorem~\ref{theorem:vacuum-1}.

Two more remarks are in order. First, it is not difficult to check that the somewhat technical conditions~\eqref{relation-theta-beta},~\eqref{thetacondition} are valid if $\rho$ is bounded away from zero and is contained in the Besov space $B^{\alpha,\infty}_q$ for some $\alpha>1/\gamma$ (in particular, $C^\alpha$ would be sufficient). 
But of course the point of this paper is that~\eqref{relation-theta-beta},~\eqref{thetacondition} allow for certain vacuum states as well, as demonstrated in the final Section~\ref{Example}. There, following the presentation of Serre~\cite{Serre}, we construct up to a finite time a fairly smooth solution of the isentropic Euler equations on $\R^2$ whose density has compact support; this solution, however, qualifies as a strong solution according to Definition~\ref{def:regular-soln-1}, and thus satisfies the weak-strong uniqueness principle. This shows that our notion of strong solution, despite its very technical definition, is not merely academic.

Secondly, on domains with physical boundaries one would not expect weak-strong uniqueness, not even in absence of vacuum. For the \emph{incompressible} Euler equations, an example has been given of a perfectly smooth solution on a perfectly smooth bounded domain with initial data that, however, admits infinitely many admissible weak solutions~\cite{BSW}, see also~\cite[Section 5]{Wiedemann}. In the compressible case, it appears to be an open question whether a similar counterexample exists, although we would expect this to be the case.

\section{Commutator estimate}
\begin{lemma}\label{lemma:commutator}
	Let $q\geq2$. Let $\mathcal{O}\subset\R^{m}$ be an open bounded set. Let $\mathcal{O}_1\subset \R^m$ be another bounded open set containing the closure of $\mathcal{O}$. Let $f\in  B^{\al_1,\f}_{q}(\mathcal{O}_1,[0,\f))$ and $g\in B^{\al_2,\f}_{q}(\mathcal{O}_1,\R^k)$ for $\al_1,\al_2\in(0,1)$, and assume $f$ and $g$ are both bounded. Let $G:[0,\f)\times\R^N\rr\R$ be a $C^1$ function such that $D_fG,D_gG$ satisfy the following properties
	
	\begin{enumerate}
		\item (H\"older continuity of $D_fG$ in $f$ variable:)
		\begin{equation*}
		\sup\limits_{\abs{\tilde g}\leq \|g\|_\infty}\sup\limits_{|f_1|,|f_2|\leq \|f\|_\infty,f_1\neq f_2}\frac{\abs{D_fG(f_1,\tilde g)-D_fG(f_2,\tilde g)}}{\abs{f_1-f_2}^\eta}\leq C
		\end{equation*}
		for $\eta\in(0,1]$ and some $C>0$ depending only on $\|f\|_\infty, \|g\|_\infty$.
		
		\item The map $g\mapsto D_fG(f,g)$ is $C^1$ uniformly in $f$ and  $D_gG\in C^1([0,\f)\times \R^N)$.
	\end{enumerate}
	 Then we have
	\begin{equation*}
	\|\nabla_x(G(f,g)_\e-G(f_\e,g_\e))\|_{L^{q/2}(\mathcal{O})}\leq C_1\left(\e^{\al_1(1+\eta)-1}+\e^{2\al_2-1}\right),
	\end{equation*}
	where $C_1$ depends only on $\abs{f}_{B^{\al_1,\f}_q}$, $\abs{g}_{B^{\al_2,\f}_q}$, $\|f\|_\infty$, and $\|g\|_\infty$.
\end{lemma}

\begin{proof}Consider the following rearrangement
	\begin{align}
	\nabla_z(G(f,g)_\e-G(f_\e,g_\e))&=\underbrace{\nabla_zG(f,g)_\e(z)-D_fG(f(z),g(z))\nb_zf_\e-D_gG(f(z),g(z))\nb_zg_\e}_{I}\nonumber\\
	&+\underbrace{D_fG(f(z),g(z))\nb_zf_\e-D_fG(f_\e(z),g_\e(z))\nb_zf_\e}_{II}\nonumber\\
	&+\underbrace{D_gG(f(z),g(z))\nb_zg_\e-D_gG(f_\e(z),g_\e(z))\nb_zg_\e}_{III}.\label{eq:lemma-1}
	\end{align}
	We first consider the term $I$,
	\begin{align}
	&\nabla_zG(f,g)_\e(z)-D_fG(f(z),g(z))\nb_zf_\e-D_gG(f(z),g(z))\nb_zg_\e\nonumber\\
	&=\int\limits_{B_\e(0)}\left[G(f(z-y),g(z-y))-D_fG(f(z),g(z))(f(z-y)-f(z))\right]\nabla\eta_\e(y)dy\nonumber\\
	&+\int\limits_{B_\e(0)}\left[-D_gG(f(z),g(z))(g(z-y)-g(z))-G(f(z),g(z))\right]\nabla\eta_\e(y)\,dy.\label{term:1:eq:lemma-1}
	\end{align}
	Let $\De_f(z,y):=f(z-y)-f(z)$ and $\De_g(z,y):=g(z-y)-g(z)$. Then, we note that
	\begin{align}
	&G(f(z-y),g(z-y))-D_fG(f(z),g(z))\De_f(z,y)-D_gG(f(z),g(z))\De_g(z,y)-G(f(z),g(z))\nonumber\\
	&=G(f(z-y),g(z-y))-G(f(z),g(z-y))-D_fG(f(z),g(z-y))\De_f(z,y)\nonumber\\
	&+D_fG(f(z),g(z-y))\De_f(z,y)-D_fG(f(z),g(z))\De_f(z,y)\nonumber\\
	&+G(f(z),g(z-y))-G(f(z),g(z))-D_gG(f(z),g(z))\De_g(z,y)\nonumber\\
	&=\De_f(z,y)\int\limits_{0}^{1}[D_{f}G(f(z)+\theta\De_f(z,y),g(z-y))-D_fG(f(z),g(z-y))]\,d\theta\nonumber\\
	&+\De_f(z,y)\int\limits_{0}^{1} D_gD_{f}G(f(z),g(z)+\theta\De_g(z,y))\cdot\De_g(z,y)\,d\theta\nonumber\\
	&+\int\limits_{0}^{1}\theta D_{gg}G(f(z),g(z)+(1-\theta)\De_g(z,y)):\De_g(z,y)\otimes\De_g(z,y)\,d\theta.\nonumber
	\end{align}
	Therefore, we have
	\begin{align*}
	&\left|G(f(z-y),g(z-y))-D_fG(f(z),g(z))\De_f(z,y)\right.\\
	&\left.-D_gG(f(z),g(z))\De_g(z,y)-G(f(z),g(z))\right|\\
	&\leq C_1\left[\abs{\De_f(z,y)}^{1+\eta}+\abs{\De_f(z,y)}\abs{\De_g(z,y)}+\abs{\De_g(z,y)}^2\right],
	\end{align*}
	We take the $L^{q/2}$ norm of both sides and use Jensen's inequality to obtain
	\begin{align}
	&\|\left(\nabla_zG(f,g)_\e(z)-D_fG(f(z),g(z))\nb_zf_\e-D_gG(f(z),g(z))\nb_zg_\e\right)\|_{L^{q/2}(\mathcal{O})}\nonumber\\
	&\leq C_1\left[\e^{-1}\sup\limits_{\abs{y}\leq \e}\left(\norm{\abs{\De_f(z,y)}^{\frac{q}{2}(1+\eta)}}_{L^1}^{\frac{2}{q}}+\norm{\De_f(z,y)}^2_{L^{q}}+\norm{\abs{\De_g(z,y)}}^2_{L^{q}}\right)\right]\nonumber\\
	&\leq C_1\left[\e^{(1+\eta)\al_1-1}+\e^{2\al_2-1}\right]\label{term:2:eq:lemma-1}
	\end{align}
	(here and in the rest of the proof, the value of $C_1$ may change from line to line.)
	Now, we estimate the terms $II$. By using H\"older's inequality, we have
	\begin{align}
	&\norm{D_fG(f(z),g(z))\nb_zf_\e-D_fG(f_\e(z),g_\e(z))\nb_zf_\e}_{L^{q/2}(\mathcal{O})}\nonumber\\
	&\leq \norm{ D_fG(f(z),g(z))\nb_zf_\e-D_fG(f_\e(z),g(z))\nb_zf_\e}_{L^{q/2}(\mathcal{O})}\nonumber\\
	&+\norm{D_fG(f_\e(z),g(z))\nb_zf_\e-D_fG(f_\e(z),g_\e(z))\nb_zf_\e}_{L^{q/2}(\mathcal{O})}\nonumber\\
	&\leq C_1\left[\norm{\abs{f-f_\e}^{\eta}}_{L^q(\mathcal{O})}+\norm{g-g_\e}_{L^q(\mathcal{O})}\right]\norm{\nabla f_\e}_{L^q(\mathcal{O})}\nonumber\\
	&\leq C_1\left[\e^{\al_1(1+\eta)-1}+\e^{\al_1+\al_2-1}\right]\leq C_1\left[\e^{\al_1(1+\eta)-1}+\e^{2\al_2-1}\right].\label{term:3:eq:lemma-1}
	\end{align}
	Finally, we estimate the term $III$. The argument is similar to \eqref{term:2:eq:lemma-1} and \eqref{term:3:eq:lemma-1}, but somewhat simpler. Indeed, from 
	\begin{align*}
	D_gG(f(z),g(z))-D_gG(f_\e(z),g_\e(z))&=\int_0^1 D_fD_gG(f(z)+s(f_\e(z)-f(z)),g(z))(f_\e(z)-f(z))ds\\
	&\quad\quad+\int_0^1 D_{gg}G(f_\e(z),g(z)+s(g_\e(z)-g(z)))(g_\e(z)-g(z))ds
	\end{align*}	
	we have the pointwise estimate
	\begin{align}
	\left|D_gG(f(z),g(z))-D_gG(f_\e(z),g_\e(z))\right|\leq C_1(|f_\e(z)-f(z)|+|g_\e(z)-g(z)|).
	\end{align}
	%
	Thus, by H\"older's inequality,
	\begin{align}
	&\norm{\left(D_gG(f,g)\nb_zg_\e-D_gG(f_\e,g_\e)\nb_zg_\e\right)}_{L^{q/2}(\mathcal{O})}\nonumber\\
	&\leq C_1\left[\norm{f-f_\e}_{L^q(\mathcal{O})}+\norm{g-g_\e}_{L^q(\mathcal{O})}\right]\norm{\nabla_z g_\e}_{L^q(\mathcal{O})}\nonumber\\
	&\leq C_1\left[\e^{\al_1+\al_2-1}+\e^{2\al_2-1}\right]\left[\abs{f}_{B^{\al_1,\f}_q}+\abs{g}_{B^{\al_2,\f}_q}\right].\label{term:4:eq:lemma-1}
	\end{align}
	Combining \eqref{term:2:eq:lemma-1}, \eqref{term:3:eq:lemma-1} and \eqref{term:4:eq:lemma-1} we conclude Lemma \ref{lemma:commutator}.

\end{proof}

\section{Proof of Theorem~\ref{theorem:vacuum-1}}\label{mainproof}

For the isentropic Euler system we consider the following relative entropy, as introduced by Dafermos~\cite{DafermosBook}: 
\begin{equation}\label{relative:isentropic-1}
\mathcal{E}\left(\vr,\U|r,\V\right):=\frac{1}{2}\vr|\U-\V|^2+H(\vr)-H^{\p}(r)(\vr-r)-H(r).
\end{equation}
We approximate $\mathcal{E}$ as follows
\begin{equation}\label{relative:isentropic-approx}
\mathcal{E}_{\si}\left(\vr,\U|r,\V\right):=\frac{1}{2}\vr|\U-\V|^2+H(\vr)-H_{\si}^{\p}(r)(\vr-r)-H_{\si}(r)
\end{equation}
where $H_{\si}$ is a $C^2$ approximation of $H$ in $C^1$, that is, $H_\si\in C^2[0,\f)$ and satisfies 
\begin{align}
\sup\{\abs{H^\p(z)-H^\p_\si(z)}+\abs{H(z)-H_\si(z)};\,z\in[0,\f)\}&\leq A_1 \si,\\
 \sup\{\abs{H^{\p\p}_\si(z)};\,z\in[0,\f)\}&\leq A_2\si^{\frac{\ga-2}{\min\{\ga,2\}-1}}\mbox{ for all }\si>0
\end{align}
for some $A_1,A_2>0$ independent of $\si$. We consider $p_\si\in C^2$ is a $C^1$ approximation of $p$ satisfying the following for some $B>0$.
\begin{align}
p^{\p}_\si(z)=zH_\si^{\p\p}(z),&\mbox{ for }z\in[0,\f),\\
\sup\{\abs{p^\p(z)-p^\p_\si(z)}+\abs{p(z)-p_\si(z)};\,z\in[0,\f)\}\leq B \si,&\mbox{ for all }\si>0.
\end{align}
We have the following proposition from~\cite{FGJ,FK}:
\begin{proposition}\label{Proposition:isen-1}
	Let $(\vr,\U)$ be an admissible weak solution to the system (\ref{isen1})--(\ref{isen2}). Let $(r,\V)\in C^1([0,T]\times{\Omega})$. Then the following holds 
	\begin{align}
	&\dom\mathcal{E}_\si(\vr,\U|r,\V)(\cdot,\tau_2)dx-\dom\mathcal{E}_\si(\vr,\U|r,\V)(\cdot,\tau_1)dx\nonumber\\
	&\leq\domm \vr(\V-\U)\cdot\pa_t\V+\vr\U\cdot\nb_x\V\cdot(\V-\U)\,dxdt\nonumber\\
	&-\domm p(\vr)\dv_x\V +(\vr-r)H_\si^{\p\p}(r)\pa_tr+\varrho H_\si^{\p\p}(r)\U\cdot\nabla_xr  \,dxdt\label{rel-ent-ineq-1}
	\end{align}
	for $0\leq \tau_1<\tau_2\leq T$.
\end{proposition}

\begin{proof}[Proof of Theorem \ref{theorem:vacuum-1}]

	
	We divide the proof into three steps.

	\noi\textbf{Step 1.} First we mollify the system \eqref{isen1}--\eqref{isen2} with the same mollifier $\eta_\e$ as in Definition \ref{def:regular-soln-1} and obtain
	\begin{align}
	\pa_tr_\e +\dv_x(r_\e\V_\e)&=\mathcal{R}_1^\e,\label{eq:isen1m}\\
	\pa_t(r_\e\V_\e)+\dv_x(r_\e\V_\e\otimes\V_\e)+\nb_xp(r_\e)&=\mathcal{R}^\e_2,\label{eq:isen2m}
	\end{align}
	where $\mathcal{R}_1^\e$ and $\mathcal{R}_2^\e$ are defined as 
	\begin{align}
	\mathcal{R}^\e_1&=\dv_x(r_\e\V_\e)-\dv_x(r\V)_\e,\\
	\mathcal{R}^\e_2&=\pa_t(r_\e\V_\e)-\pa_t(r\V)_\e+\dv_x(r_\e\V_\e\otimes\V_\e)-\dv_x(r\V\otimes\V)_\e+\nb_xp(r_\e)-\nb_xp(r)_\e.
	\end{align}
	After a modification of \eqref{eq:isen1m} and \eqref{eq:isen2m} we have
	\begin{align}
	\pa_tr_\e&=-\dv_x(r_\e\V_\e)+\mathcal{R}_1^\e,\label{def:r_e}\\
	r_\e\pa_t\V_\e&=-r_\e\V_\e\cdot\nabla\V_\e-r_\e H^{\p\p}(r_\e)\nabla_xr_\e+\mathcal{R}_2^\e-\mathcal{R}_1^\e\V_\e.\label{def:v_e}
	\end{align}
	We modify \eqref{def:v_e} as
	\begin{equation}
		r_\e^\de\pa_t\V_\e=-r^\de_\e\V_\e\cdot\nabla\V_\e-r^\de_\e H^{\p\p}_\si(r_\e)\nabla_xr_\e+\mathcal{R}_2^\e-\mathcal{R}_1^\e\V_\e+\mathcal{M}_\e^\de,\label{def:v-1}
	\end{equation}
	where $r_\e^\de$ will be defined in~\eqref{deltadef} below and
	\begin{equation}\label{def:M}
	\mathcal{M}_\e^\de:=(r_\e^\de-r_\e)\left[\pa_t\V_\e+\V_\e\cdot\nabla\V_\e+H^{\p\p}_\si(r_\e)\nabla_xr_\e\right]+( p^{\p}_\si(r_\e)-p^{\p}(r_\e))\nabla_xr_\e.
	\end{equation}
	\noi\textbf{Step 2.} Now we invoke Proposition \ref{Proposition:isen-1} with $r=r_\e$ and $\V=\V_\e$ to get
	\begin{align}
	& \int\limits_{{\Omega}}\mathcal{E}_\sigma(\vr,\U|r_\e,\V_\e)(\tau_2,\cdot)dx-  \int\limits_{{\Omega}}\mathcal{E}_\sigma(\vr,\U|r_\e,\V_\e)(\tau_1,\cdot)dx\nonumber\\
	&\leq\int\limits_{\tau_1}^{\tau_2} \int\limits_{{\Omega}} \vr(\V_\e-\U)\cdot\pa_t\V_\e+\vr\U\cdot\nb_x\V_\e\cdot(\V_\e-\U)\,dxdt\nonumber \\
	&-\int\limits_{\tau_1}^{\tau_2} \int\limits_{{\Omega}} p(\vr)\dv_x\V_\e +{(\vr-r_\e)}H_\si^{\p\p}(r_\e)\pa_tr_\e +{\vr}H_\si^{\p\p}(r_\e)\U\cdot\nb_xr_\e\,dxdt.\label{cal0:thm}
	\end{align}
	We divide the integral into three parts and estimate separately:
	\begin{align}
	&\int\limits_{\tau_1}^{\tau_2} \int\limits_{{\Omega}} \vr(\V_\e-\U)\cdot\pa_t\V_\e+\vr\U\cdot\nb_x\V_\e\cdot(\V_\e-\U)\,dxdt\nonumber \\
	&-\int\limits_{\tau_1}^{\tau_2} \int\limits_{{\Omega}} p(\vr)\dv_x\V_\e +{(\vr-r_\e)}H_\si^{\p\p}(r_\e)\pa_tr_\e +{\vr}H_\si^{\p\p}(r_\e)\U\cdot\nb_xr_\e\,dxdt\nonumber\\
	=&\underbrace{\int\limits_{\mathcal{W}_\e[\tau_1,\tau_2]}  \vr(\V_\e-\U)\cdot\pa_t\V_\e+\vr\U\cdot\nb_x\V_\e\cdot(\V_\e-\U)\,dxdt}_{=:I_{11}}\nonumber \\
	&-\underbrace{\int\limits_{\mathcal{W}_\e[\tau_1,\tau_2]}  p(\vr)\dv_x\V_\e +{(\vr-r_\e)}H_\si^{\p\p}(r_\e)\pa_tr_\e +{\vr}H_\si^{\p\p}(r_\e)\U\cdot\nb_xr_\e\,dxdt}_{=:I_{12}}\nonumber\\
	&+\underbrace{\int\limits_{[\tau_1,\tau_2]\times\Omega\setminus \mathcal{W}_\e[\tau_1,\tau_2]}  \vr(\V_\e-\U)\cdot\pa_t\V_\e+\vr\U\cdot\nb_x\V_\e\cdot(\V_\e-\U)-p(\vr)\dv_x\V_\e\,dxdt}_{=:I_2}.
	\end{align} 
	Recall from Definition~\ref{def:regular-soln-1} that $r_\e>0$ in $\mathcal{W}_\e[\tau_1,\tau_2]$. By applying \eqref{def:r_e} and \eqref{def:v-1} in $I_{11}$ we have
	\begin{align}
	I_{11}=&\int\limits_{\mathcal{W}_\e[\tau_1,\tau_2]}  \vr(\V_\e-\U)\cdot\pa_t\V_\e+\vr\U\cdot\nb_x\V_\e\cdot(\V_\e-\U)\,dxdt\nonumber \\
	=&\int\limits_{\mathcal{W}_\e[\tau_1,\tau_2]}  -\vr(\V_\e-\U)\cdot\nb_x\V_\e\cdot(\V_\e-\U)-\vr H_\si^{\p\p}(r_\e)(\V_\e-\U)\cdot\nabla_xr_\e\,dxdt\nonumber\\
	&+\int\limits_{\mathcal{W}_\e[\tau_1,\tau_2]} \frac{1}{r^\de_\e}\vr(\V_\e-\U)\cdot[\mathcal{R}_2^\e-\mathcal{R}_1^\e\V_\e+\mathcal{M}_\e^\de]\,dxdt.\label{cal:I11:thm} 
	\end{align}
	By a similar argument, we can modify $I_{12}$ and get 
	\begin{align}
	I_{12}=&\int\limits_{\mathcal{W}_\e[\tau_1,\tau_2]}  p(\vr)\dv_x\V_\e +{(\vr-r_\e)}H_\si^{\p\p}(r_\e)\pa_tr_\e +{\vr}H_\si^{\p\p}(r_\e)\U\cdot\nb_xr_\e\,dxdt\nonumber\\
	=&\int\limits_{\mathcal{W}_\e[\tau_1,\tau_2]} \left(p(\vr)-(\vr-r_\e)p_\si^{\p}(r_\e)-p_\si(r_\e)\right)\dv_x\V_\e-{\vr}H_\si^{\p\p}(r_\e)(\V_\e-\U)\cdot\nb_xr_\e\,dxdt\nonumber\\
	&+\int\limits_{\mathcal{W}_\e[\tau_1,\tau_2]} {(\vr-r_\e)}H_\si^{\p\p}(r_\e) \mathcal{R}_1^{\e}\,dxdt.\label{cal:I12:thm}
	\end{align}
	Note that the complement of $supp(r_\e)$ is a subset of the complement of $supp(r)$. Since $\V$ satisfies \eqref{eqn:mom-vacuum-1} in the complement of $supp(r)$, we get
	\begin{align}
	I_2=&\int\limits_{[\tau_1,\tau_2]\times\Omega\setminus \mathcal{W}_\e[\tau_1,\tau_2]} \vr(\V_\e-\U)\cdot\pa_t\V_\e+\vr\U\cdot\nb_x\V_\e\cdot(\V_\e-\U)-p(\vr)\dv_x\V_\e \,dxdt\nonumber \\
	=&-\int\limits_{[\tau_1,\tau_2]\times\Omega\setminus \mathcal{W}_\e[\tau_1,\tau_2]} \vr(\V_\e-\U)\cdot\nb_x\V_\e\cdot(\V_\e-\U)+p(\vr)\dv_x\V_\e\,dxdt\nonumber\\
	&+\int\limits_{[\tau_1,\tau_2]\times\Omega\setminus \mathcal{W}_\e[\tau_1,\tau_2]} \vr(\V_\e-\U)\cdot(\V_\e\cdot\nabla_x\V_\e-(\V\cdot\nb_x\V)_\e)\,dxdt.\label{cal:I2:thm}
	\end{align}
	Combining \eqref{cal:I11:thm}, \eqref{cal:I12:thm} and \eqref{cal:I2:thm} we obtain
	\begin{align}
	& \int\limits_{{\Omega}}\mathcal{E}_\si(\vr,\U|r_\e,\V_\e)(\tau_2,\cdot)dx-  \int\limits_{{\Omega}}\mathcal{E}_\si(\vr,\U|r_\e,\V_\e)(\tau_1,\cdot)dx\nonumber\\
	&\leq \int\limits_{\tau_1}^{\tau_2} \int\limits_{{\Omega}} -\vr(\V_\e-\U)\cdot\nb_x\V_\e\cdot(\V_\e-\U)\,dxdt+\int\limits_{\mathcal{W}_\e[\tau_1,\tau_2]} \mathcal{R}_3^{\e,\si}+\frac{1}{r^\de_\e}\vr(\V_\e-\U)\cdot[\mathcal{R}_2^\e-\mathcal{R}_1^\e\V_\e+\mathcal{M}_\e^\de]\,dxdt\nonumber \\
	&-\int\limits_{\mathcal{W}_\e[\tau_1,\tau_2]} \left(p(\vr)-(\vr-r_\e)p^{\p}(r_\e)-p(r_\e)\right)\dv_x\V_\e\,dxdt-\int\limits_{[\tau_1,\tau_2]\times\Omega\setminus \mathcal{W}_\e[\tau_1,\tau_2]} p(\vr)\dv_x\V_\e \,dxdt\nonumber \\
	&-\int\limits_{\mathcal{W}_\e[\tau_1,\tau_2]} {(\vr-r_\e)}H_\si^{\p\p}(r_\e) \mathcal{R}_1^{\e}\,dxdt+\int\limits_{[\tau_1,\tau_2]\times\Omega\setminus \mathcal{W}_\e[\tau_1,\tau_2]} \vr(\V_\e-\U)\cdot(\V_\e\cdot\nabla_x\V_\e-(\V\cdot\nb_x\V)_\e)\,dxdt\label{cal:thm}
	\end{align}
	where $\mathcal{R}_3^{\e,\si}$ is defined as 
	\begin{equation*}
	\mathcal{R}_3^{\e,\si}:= \left((\vr-r_\e)(p_\si^{\p}(r_\e)-p^\p(r_\e))+(p_\si(r_\e)-p(r_\e))\right)\dv_x\V_\e.
	\end{equation*}
	\noi\textbf{Step 3.} 
	By using Definition~\ref{def:regular-soln-1}(\ref{A1-3}) in \eqref{cal:thm}, we have
	\begin{align}
	& \int\limits_{{\Omega}}\mathcal{E}_\si(\vr,\U|r_\e,\V_\e)(\tau_2,\cdot)dx-  \int\limits_{{\Omega}}\mathcal{E}_\si(\vr,\U|r_\e,\V_\e)(\tau_1,\cdot)dx\nonumber\\
	&\leq \int\limits_{\tau_1}^{\tau_2} \int\limits_{{\Omega}} \Lambda(t)\mathcal{E}(\vr,\U|r_\e,\V_\e)\,dxdt+\int\limits_{\mathcal{W}_\e[\tau_1,\tau_2]}\mathcal{R}_3^{\e, \sigma}+\frac{1}{r^\de_\e}\vr(\V_\e-\U)\cdot[\mathcal{R}_2^\e-\mathcal{R}_1^\e\V_\e+\mathcal{M}_\e^\de]\,dxdt\nonumber \\
	&-\int\limits_{\mathcal{W}_\e[\tau_1,\tau_2]}{(\vr-r_\e)}H_\si^{\p\p}(r_\e) \mathcal{R}_1^{\e}\,dxdt+\int\limits_{[\tau_1,\tau_2]\times\Omega\setminus\mathcal{W}_\e[\tau_1,\tau_2]} \vr(\V_\e-\U)\cdot(\V_\e\cdot\nabla_x\V_\e-(\V\cdot\nb_x\V)_\e)\,dxdt.\label{cal1:thm}
	\end{align}
	We wish to pass to the limit in \eqref{cal1:thm} as $\e\rr0$. Since $\V\in C^1(([0,T]\times\Omega)\setminus\mathcal{W}_0[0,T])$ and $\mathcal{W}_0\subset\mathcal{W}_\e$, we get
	\begin{align}
 &\int\limits_{[\tau_1,\tau_2]\times\Omega\setminus\mathcal{W}_\e[\tau_1,\tau_2]} \abs{\vr(\V_\e-\U)\cdot(\V_\e\cdot\nabla_x\V_\e-(\V\cdot\nb_x\V)_\e)}\,dxdt\nonumber\\
  &\leq \int\limits_{[\tau_1,\tau_2]\times\Omega\setminus\mathcal{W}_0[\tau_1,\tau_2]} \abs{\vr(\V_\e-\U)\cdot(\V_\e\cdot\nabla_x\V_\e-(\V\cdot\nb_x\V)_\e)}\,dxdt\rr0\mbox{ as }\e\rr0.\label{cal2:thm}
	\end{align}
	Next, we want to prove
	\begin{align}
	\int\limits_{\mathcal{W}_\e[\tau_1,\tau_2]}\mathcal{R}_3^{\e,\si}+\frac{1}{r^\de_\e}\vr(\V_\e-\U)\cdot[\mathcal{R}_2^\e-\mathcal{R}_1^\e\V_\e+\mathcal{M}_\e^\de]\,dxdt&\rr0\mbox{ as }\e\rr0,\label{cal4:thm}\\
		\int\limits_{\mathcal{W}_\e[\tau_1,\tau_2]}{(\vr-r_\e)}H_\si^{\p\p}(r_\e) \mathcal{R}_1^{\e}\,dxdt&\rr0\mbox{ as }\e\rr0.\label{cal3:thm}
	\end{align}
	We first show \eqref{cal4:thm}. To this end we use H\"older's inequality to have
	\begin{align}
	&\abs{	\int\limits_{\mathcal{W}_\e[\tau_1,\tau_2]}\frac{1}{r^\de_\e}\vr(\V_\e-\U)\cdot[\mathcal{R}_2^\e-\mathcal{R}_1^\e\V_\e+\mathcal{M}_\e^\de]\,dxdt}\\
	&\leq \|(r^\de_\e)^{-1}\left(\mathcal{R}_2^\e+\mathcal{M}_\e^\de\right)\|_{L^s(\mathcal{W}_\e[\tau_1,\tau_2])}\|\vr(\V_\e-\U)\|_{L^{s'}(\mathcal{W}_\e[\tau_1,\tau_2])}\nonumber\\
	&+\|(r^\de_\e)^{-1}\mathcal{R}_1^\e\|_{L^s(\mathcal{W}_\e[\tau_1,\tau_2])}\|\vr(\abs{\V_\e}^2-\V\cdot\U)\|_{L^{s'}(\mathcal{W}_\e[\tau_1,\tau_2])},\nonumber
	\end{align}
	where $s=2\ga/(\ga-1)$ and $s'=2\ga/(\ga+1)$. By Young's inequality we get
	\begin{equation*}
	\varrho^{\frac{2\ga}{\ga+1}}\abs{\textbf{u}}^{\frac{2\ga}{\ga+1}}\leq C(\ga)\left[\varrho^{\ga}+\varrho\abs{\U}^2\right].
	\end{equation*}
	By using \eqref{energy-inequality} we have
	\begin{equation*}
	\|\vr(\V_\e-\U)\|_{L^{s'}(\mathcal{W}_\e[\tau_1,\tau_2])}\leq C(\tau_1,\tau_2)(1+\|\V\|_{L^{\f}([\tau_1,\tau_2]\times\Omega)})^2\|\varrho_0^\ga+\varrho_0\abs{\U_0}^2\|_{L^{1}(\Omega)}^{\frac{\ga+1}{2\ga}}
	\end{equation*}
	where $C(\tau_1,\tau_2)$ depends on the volume of $[\tau_1,\tau_2]\times\Omega$. By using the definition of $\mathcal{M}_\e^\de$, \eqref{def:M} we obtain the following
	\begin{align*}
	\norm{\frac{1}{r^\de_\e}\mathcal{M}_\e^\de}_{L^{s}(\mathcal{W}_\e[\tau_1,\tau_2])}&\leq C_0\norm{\frac{r_\e^\de-r_\e}{r^\de_\e}}_{L^{2s}(\mathcal{W}_\e[\tau_1,\tau_2])}\norm{\left(\abs{\pa_t\textbf{v}_\e}+\abs{\nabla\textbf{v}_\e}\right)}_{L^{2s}(\mathcal{W}_\e[\tau_1,\tau_2])}\\
	&+C_0\si^{\frac{\ga-2}{\min\{\ga,2\}-1}}\norm{\frac{r_\e^\de-r_\e}{r^\de_\e}}_{L^{2s}(\mathcal{W}_\e[\tau_1,\tau_2])}\norm{\nabla r_\e}_{L^{2s}(\mathcal{W}_\e[\tau_1,\tau_2])}\nonumber\\
	&+C_1\si\norm{\nabla r_\e}_{L^{2s}(\mathcal{W}_\e[\tau_1,\tau_2])}.
	\end{align*}
	Now we set 
	\begin{equation}\label{deltadef}
	r_\e^\de:=r_\e\left(1+\frac{\de}{ r_\e^p}\right)^{\frac{1}{\tilde q}}\mbox{ whenever }r_\e>0,
	\end{equation}
	for $\de>0, \tilde q>0,p>0$ to be chosen later. By using the inequality $1-\frac{1}{(1+x)^\B}-(1+\B)x\leq0$ for any $\B>0$ and $x\geq0$, we obtain
	\begin{equation*}
	\frac{r_\e^\de-r_\e}{r^\de_\e}=1-\frac{1}{(1+\de r_\e^{-p})^{\frac{1}{\tilde q}}}\leq \frac{C_{\tilde q}\de}{r^{p}_\e}\mbox{ if }r_\e\in(0,1].
	\end{equation*}
	Hence, we have
	\begin{align}
	\norm{\frac{1}{r^\de_\e}\mathcal{M}_\e^\de}_{L^{s}(\mathcal{W}_\e[\tau_1,\tau_2])}
	&\leq C_{\tilde q}\de\norm{\frac{1}{r_\e^{2sp}}}_{L^1(\mathcal{W}_\e[\tau_1,\tau_2])}^{\frac{1}{2s}}\norm{\left(\abs{\pa_t\textbf{v}_\e}+\abs{\nabla\textbf{v}_\e}\right)}_{L^{2s}(\mathcal{W}_\e[\tau_1,\tau_2])}\nonumber\\
	&+C_{\tilde q}\left(\de\si^{\frac{\ga-2}{\min\{\ga,2\}-1}}\norm{\frac{1}{r_\e^{2sp}}}_{L^1(\mathcal{W}_\e[\tau_1,\tau_2])}^{\frac{1}{2s}}+\si\right)\norm{\nabla{r}_\e}_{L^{2s}(\mathcal{W}_\e[\tau_1,\tau_2])}.
	\end{align}
	Since $r\in B^{\B,\f}_q$ and $\textbf{v}\in B^{\al,\f}_q$, and since $q\geq s$ and therefore $B^{\alpha,\infty}_{2q}\subset B^{\alpha,\infty}_{2s}$ and similarly with $B^{\beta,\infty}_{2q}$, we get
	\begin{align}
	\norm{\frac{1}{r^\de_\e}\mathcal{M}_\e^\de}_{L^{s}(\mathcal{W}_\e[\tau_1,\tau_2])}
	& \leq C_{\tilde q} \de\e^{\al-1}\norm{\frac{1}{r^{2sp}_\e}}_{L^1(\mathcal{W}_\e[\tau_1,\tau_2])}^{\frac{1}{2s}}\abs{\textbf{v}}_{B^{\al,\f}_{2q}}\nonumber\\
		&+C_{\tilde q}\left(\de\si^{\frac{\ga-2}{\min\{\ga,2\}-1}}\e^{\B-1}\norm{\frac{1}{r_\e^{2sp}}}_{L^1(\mathcal{W}_\e[\tau_1,\tau_2])}^{\frac{1}{2s}}+\si\e^{\B-1}\right)\abs{r}_{B^{\B,\f}_{2q}}.\label{eq:exp-1}
	\end{align}
	By H\"older's inequality, we obtain
	\begin{align*}
	\|(r^\de_\e)^{-1}\mathcal{R}_2^\e\|_{L^s(\mathcal{W}_\e[\tau_1,\tau_2])}
	&\leq \norm{\frac{1}{r^\de_\e}}_{L^{2s}(\mathcal{W}_\e[\tau_1,\tau_2])}\norm{\mathcal{R}_2^\e}_{L^{2s}(\mathcal{W}_\e[\tau_1,\tau_2])}.
	\end{align*}
	Note that
	\begin{align}
	\frac{1}{r_\e^\de}=\frac{1}{r_\e(1+\de r^{-p}_\e)^{\frac{1}{\tilde q}}}\leq \frac{\de^{-\frac{1}{\tilde q}}}{r_\e^{\frac{\tilde q-p}{\tilde q}}}.
	\end{align}
	Hence, by Lemma \ref{lemma:commutator} we have
	\begin{align}
	\|(r^\de_\e)^{-1}\mathcal{R}_2^\e\|_{L^s(\mathcal{W}_\e[\tau_1,\tau_2])}\leq C_1\de^{-\frac{1}{\tilde q}}(\e^{2\al-1}+\e^{\min\{\ga,2\}\B-1})\norm{\frac{1}{r^{2s\frac{\tilde q-p}{\tilde q}}_\e}}_{L^1(\mathcal{W}_\e[\tau_1,\tau_2])}^{\frac{1}{2s}}.\label{eq:exp-2}
	\end{align}
	For $\mathcal{R}^{\e,\si}_3$ we use H\"older's inequality to get
	\begin{align}
	&\int\limits_{\mathcal{W}_\e[\tau_1,\tau_2]}\abs{\left((\vr-r_\e)(p_\si^{\p}(r_\e)-p^\p(r_\e))+(p_\si(r_\e)-p(r_\e))\right)\dv_x\V_\e}\,dxdt\nonumber\\
	&\leq C_1\si\left(\|\varrho-r_\e\|_{L^{\gamma}((0,T)\times\Omega)}+1\right)\|\nabla_x\textbf{v}_\e\|_{L^{\frac{\gamma}{\gamma-1}}(\mathcal{W}_\e[\tau_1,\tau_2])}\nonumber\\
	&\leq C_1\si\e^{\al-1}\|\varrho-r_\e\|_{L^{\gamma}((0,T)\times\Omega)}.\label{eq:exp-2.5}
	\end{align} 
	The estimate for $\|(r_\e^\de)^{-1}\mathcal{R}_1^\e\|_{L^s(\mathcal W_\e[\tau_1,\tau_2])}$ is even easier and we omit the details. 
	
	Again by H\"older's inequality and Lemma \ref{lemma:commutator} we prove 
	\begin{align}
	&\int\limits_{\mathcal{W}_\e[\tau_1,\tau_2]}\abs{(\vr-r_\e)H_\si^{\p\p}(r_\e)\mathcal{R}_1^\e}\,dxdt\nonumber\\
	&\leq C_1\si^{\frac{\ga-2}{\min\{\ga,2\}-1}}\|\varrho-r_\e\|_{L^{\gamma}((0,T)\times\Omega)}\|\mathcal{R}^\e_1\|_{L^{\frac{\gamma}{\gamma-1}}(\mathcal{W}_\e[\tau_1,\tau_2])}\nonumber\\
	&\leq C_1\si^{\frac{\ga-2}{\min\{\ga,2\}-1}}(\e^{2\al-1}+\e^{2\B-1})\|\varrho-r_\e\|_{L^{\gamma}((0,T)\times\Omega)}.\label{eq:exp-3}
	\end{align} 
	Now we set $\de=\e^{\kappa}$, $\si=\e^{\nu}$ for some $\kappa,\nu>0$ and $p, \tilde q$ as follows:
	\begin{equation*}
	p=\frac{\ga-1}{4\ga}\theta\mbox{ and }\tilde q=\left[\frac{4\ga}{\ga-1}\frac{1}{\theta}-1\right]^{-1}.
	\end{equation*} 
	We wish to pass to the limit as $\e\rr0$ in \eqref{eq:exp-1}, \eqref{eq:exp-2}, \eqref{eq:exp-2.5} and \eqref{eq:exp-3}. To this end we need $\al,\B,\kappa,\nu$ to satisfy the following%
	\begin{align}
	\kappa+\al-1>0,\\
	\kappa+\frac{\ga-2}{\min\{\ga,2\}-1}\nu+\B-1>0,\\
	\nu+\B-1>0,\\
	-\frac{\kappa}{q}+2\al-1>0,\\
	-\frac{\kappa}{q}+\min\{\ga,2\}\B-1>0,\\
	\frac{\ga-2}{\min\{\ga,2\}-1}\nu+2\al-1>0,\\
	\frac{\ga-2}{\min\{\ga,2\}-1}\nu+2\B-1>0.
	\end{align}
	For $\al\geq\B$, it is suffices to get
	\begin{align}
	\kappa+\B-1>0,\label{alg-eq-1}\\
	\kappa+\frac{\ga-2}{\min\{\ga,2\}-1}\nu+\B-1>0,\label{alg-eq-2}\\
	\nu+\B-1>0,\label{alg-eq-3}\\
	-\frac{\kappa}{q}+\min\{\ga,2\}\B-1>0,\label{alg-eq-4}\\
	\frac{\ga-2}{\min\{\ga,2\}-1}\nu+2\B-1>0.\label{alg-eq-5}
	\end{align}
	
	{Note that for $1<\ga<2$ if $\theta,\B$ satisfy the relation $\theta>\frac{4\ga^2}{(\ga-1)((\ga-1)^2\B+1-\B)}(1-\B)$ and $\B> \frac{1}{\ga}$, then we can choose $\kappa,\nu$ as follows:
	\begin{equation*}
	\begin{array}{rll}
	 	\frac{\ga}{\ga-1}(1-\B)&<\kappa&<\left[\frac{4\ga}{\ga-1}\frac{1}{\theta}-1\right]^{-1}(\ga\B-1),\\
	 	1-\B&<\nu&<\min\left\{\frac{\ga-1}{2-\ga}(2\B-1),\frac{1}{2-\gamma}(1-\B)\right\}.
	\end{array}
	\end{equation*}
Since $\kappa>\frac{\ga}{\ga-1}(1-\B)>1-\B$, we get \eqref{alg-eq-1}. Similarly, we have
\begin{align*}
\kappa+\frac{\gamma-2}{\ga-1}\nu+\B-1&>\frac{\ga}{\ga-1}(1-\B)+\frac{\gamma-2}{\ga-1}\nu-(1-\B)\\
&=\frac{1}{\ga-1}(1-\B)-\frac{2-\gamma}{\ga-1}\nu\\
&>\frac{1}{\ga-1}[(1-\B)-(2-\ga)\nu]>0.
\end{align*}
    For $\ga\geq2$, we observe that \eqref{alg-eq-2}, \eqref{alg-eq-5} follows from \eqref{alg-eq-1} and \eqref{alg-eq-3} respectively. Since for $\ga\geq2$, parameters $\theta,\B$ satisfy $\theta>\frac{4\ga}{\ga-1}\frac{1-\B}{\B}$ and $\B>\frac{1}{2}$,} we can choose $\kappa,\nu$ as follows:
	\begin{equation*}
	1-\B<\kappa<\left[\frac{4\ga}{\ga-1}\frac{1}{\theta}-1\right]^{-1}(2\B-1)\mbox{ and }1-\B<\nu<\f.
	\end{equation*}
	This completes the proof of Theorem \ref{theorem:vacuum-1}.

\end{proof}

\section{An Explicit Example}\label{Example}
In this section, we give a fairly explicit example of quite smooth initial data with vacuum outside a ball so that the evolving strong solution satisfies the integrability condition \eqref{relation-theta-beta} and therefore weak-strong uniqueness.

Let $s>2$. We take initial data $\varrho_0,\textbf{u}_0$ satisfying the following:
\begin{enumerate}
	\item The density function $\varrho_0\in C^1(\R^2,[0,\f))$ and 
	\begin{equation*}
	\varrho_0>0\mbox{ in }B(\textbf{0},R)\mbox{ and }\varrho_0=0\mbox{ in }B(\textbf{0},R)^c
	\end{equation*}
	for some $R>0$. Also, $c_0:=\varrho_0^{\frac{\gamma-1}{2}}\in H^s(\R^2)$. Additionally, we assume that
	\begin{equation*}
	\int\limits_{B(\textbf{0},R)}\frac{1}{\varrho_0^\theta}\,dx\leq C_0\mbox{ for some }\theta>0\mbox{ and }C_0>0. 
	\end{equation*}
	
	\item The velocity component $\textbf{u}_0\in C^1(\R^2,\R^2)\cap H^s(\R^2)$. We further assume that
	\begin{equation*}
	\textbf{u}_0(x)=x\mbox{ when }\abs{x}=R.
	\end{equation*}
\end{enumerate}
It is not difficult to see that such $\rho_0$ exists. Indeed, depending on $\gamma$, choose sufficiently large $N>0$ and $\theta<\frac1N$, and let $\varrho_0(x)=(|x|-R)^N$ for $|x|$ smaller than but close to $R$, and $\varrho_0=0$ outside $B(\textbf{0},R)$.

 Now, by \cite{Kato} there exists a solution pair $(\varrho,\textbf{u})\in C(0,T;H^s(\R^2))\cap C^1(0,T;H^{s-1}(\R^2))$ for some $T>0$. Let $\Omega(t):=\{\varrho(t)>0\}$ and $\Gamma(t)$ be the vacuum boundary, that is, $\Gamma(t)=\pa\Omega(t)$. It has been proved in \cite{LiuYang,Serre} (more precisely see \cite[Proposition 2.1]{Serre}) that 
\begin{equation*}
\Gamma(t)=\Psi_t(\Gamma(0))\mbox{ where }\Psi_t(x)=x+t\textbf{u}_0(x).
\end{equation*}
By the choice of $\varrho_0,\textbf{u}_0$ we get $\Gamma(t)=\{(1+t)x;\abs{x}=R\}=\{x;\,\abs{x}=(1+t)R\}$. Suppose $X(t)$ solves
\begin{equation*}
\frac{dX(t,x_0)}{dt}=\textbf{u}(t,X(t,x_0))\mbox{ and }X(0,x_0)=x_0.
\end{equation*}
Since along $\Gamma(t)$, the velocity component $\textbf{u}$ verifies $\pa_t\textbf{u}+\textbf{u}\cdot\nabla\textbf{u}=0$, we have $\textbf{u}(t,X(t,x_0))=\textbf{u}_0(x_0)$ for $x_0\in\Gamma(0)$ (see \cite{Serre}). Hence we have that $(t,X(t,x_0))$ is a straight line in $\R_+\times\R^2$. Therefore, the velocity of $\Gamma(t)$ at time $\tau$ is $R\xi$ where $\xi=\Gamma(t)/\abs{\Gamma(t)}$. Hence, we get
\begin{align*}
\frac{d}{dt}\int\limits_{\Omega(t)}\frac{1}{(\e+\varrho)^\theta}\,dx&=R\int\limits_{\pa\Omega(t)}\frac{1}{(\e+\varrho)^\theta}\,dS-\theta\int\limits_{\Omega(t)}\frac{\pa_t\varrho}{(\e+\varrho)^{1+\theta}}\,dx\\
&=R\int\limits_{\pa\Omega(t)}\frac{1}{(\e+\varrho)^\theta}\,dS+\theta\int\limits_{\Omega(t)}\frac{\dv_x(\varrho\textbf{u})}{(\e+\varrho)^{1+\theta}}\,dx\\
&=R\int\limits_{\pa\Omega(t)}\frac{1}{(\e+\varrho)^\theta}\,dS+\theta\int\limits_{\Omega(t)}\frac{\varrho\dv_x\textbf{u}}{(\e+\varrho)^{1+\theta}}\,dx-\int\limits_{\Omega(t)}\textbf{u}\cdot\nabla\left({\frac{1}{(\e+\varrho)^{\theta}}}\right)\,dx.
\end{align*}
After using integration by parts in the last integral, we obtain
\begin{align*}
\frac{d}{dt}\int\limits_{\Omega(t)}\frac{1}{(\e+\varrho)^\theta}\,dx&=R\int\limits_{\pa\Omega(t)}\frac{1}{(\e+\varrho)^\theta}\,dS+\theta\int\limits_{\Omega(t)}\frac{\varrho\dv_x\textbf{u}}{(\e+\varrho)^{1+\theta}}\,dx+\int\limits_{\Omega(t)}\frac{\dv_x\textbf{u}}{(\e+\varrho)^{1+\theta}}\,dx\\
&-\int\limits_{\pa\Omega(t)}\frac{1}{(\e+\varrho)^\theta}\textbf{u}\cdot\nu\,dS\\
&=R\int\limits_{\pa\Omega(t)}\frac{1}{(\e+\varrho)^\theta}\,dS+\theta\int\limits_{\Omega(t)}\frac{\varrho\dv_x\textbf{u}}{(\e+\varrho)^{1+\theta}}\,dx+\int\limits_{\Omega(t)}\frac{\dv_x\textbf{u}}{(\e+\varrho)^{1+\theta}}\,dx\\
&-R\int\limits_{\pa\Omega(t)}\frac{1}{(\e+\varrho)^\theta}\,dS\\
&\leq C_\theta\norm{\textbf{u}}_{C^1([0,T]\times\R^2)}\int\limits_{\Omega(t)}\frac{1}{(\e+\varrho)^\theta}\,dx.
\end{align*}
In the previous calculation, the second equality follows from the fact that $\textbf{u}(t,X(t,x_0))=u_0(x_0)$. By using Gr\"{o}nwall's inequality and passing to the limit as $\e\rr0$, we get
\begin{equation*}
\int\limits_{\Omega(t)}\frac{1}{\varrho^\theta}\,dx\leq C_1\int\limits_{\Omega(0)}\frac{1}{\varrho_0^\theta}\,dx,
\end{equation*}
where $C_1$ depends on $C_\theta$ and $\norm{\textbf{u}}_{C^1([0,T]\times\R^2)}$. Suppose $0\leq \eta\in C^1(\R\times\R^2)$ such that 
{
	\begin{equation*}
supp[\eta]=[0,1]\times\overline{B(\textbf{0},1)}\mbox{ and }	\int\limits_{[0,1]}\int\limits_{B(\textbf{0},1)}\frac{1}{\eta^\theta}\,dyds\leq C_\eta<\f.
\end{equation*}
We extend $\varrho$ beyond $[0,T]\times\R^2$ by $0$, that is,
\begin{equation}
\bar{\varrho}(t,x):=\left\{\begin{array}{ll}
\varrho(t,x)&\mbox{ if }t\in[0,T],\,x\in\R^2,\\
0&\mbox{ otherwise. }
\end{array}\right.
\end{equation}
We define $\varrho_\e:=\bar{\varrho}*\eta_\e$. Then, we observe that
\begin{align*}
\varrho_\e(t,x)=\bar{\varrho}*\eta_\e(t,x)&=\int\limits_{[0,\e]}\int\limits_{\abs{y}\leq \e}\frac{1}{\e^2}\eta(s/\e,y/\e)\bar{\varrho}(t-s,x-y)\,dyds\\
&=\int\limits_{[0,1]}\int\limits_{\abs{y}\leq 1}\eta(s,y)\bar{\varrho}(t-\e s,x-\e y)\,dy.
\end{align*}
Note that $\bar{\varrho}(t-\e s,x-\e y)$ vanishes for the following two cases:
\begin{enumerate}[(a)]
	\item $t-\e s<0$ or $t-\e s>T$,
	\item $t-\e s\in[0,T]$ and $x-\e y\notin B(\textbf{0},(1+t-\e s)R)$.
\end{enumerate}
To simplify the notations, we define $\mathcal{I}_{t,\e}=[\max\{0,(t-T)/\e\},t/\e]$ and 
\begin{equation}
\Omega_{t,x,\e}(s):=\{\abs{y}\leq 1; x-\e y\in B(\textbf{0},(1+t-\e s)R)\}. 
\end{equation} Therefore,
\begin{equation*}
\varrho_\e(t,x)=\varrho*\eta_\e(t,x)=\int\limits_{\mathcal{I}_{t,\e}}\int\limits_{\Omega_{t,x,\e}(s)}\eta(s,y)\varrho(t-\e s,x-\e y)\,dyds.
\end{equation*}
In this example $\mathcal{W}_\e$ can be taken as $\bigcup\limits_{0\leq t\leq T}\{t\}\times B(\textbf{0},(1+t)R+\e)$. 
For technical purposes, let us introduce another parameter $\delta>0$, and define $\varrho_{\e,\de}$ as follows:
\begin{equation}
\varrho_{\e,\de}(t,x):=\int\limits_{\mathcal{I}_{t,\e}}\int\limits_{\Omega_{t,x,\e}(s)}(\eta(s,y)+\de)(\varrho(t-\e s,x-\e y)+\de)\,dyds.
\end{equation}
Subsequently, by using Jensen's inequality we get
\begin{align*}
&\int\limits_{\mathcal{W}_\e}\frac{1}{\varrho_{\e,\de}^\theta}\,dxdt\\
&\leq C\int\limits_{0}^T\int\limits_{B(\textbf{0},(1+t)R+\e)}\int\limits_{\mathcal{I}_{t,\e}}\int\limits_{\Omega_{t,x,\e}(s)}\frac{1}{(\eta(s,y)+\de)^\theta(\varrho(t-\e s,x-\e y)+\de)^\theta}\,dydsdxdt\\
&= C\int\limits_{[0,1]}\int\limits_{\abs{y}\leq 1}\frac{1}{(\eta(s,y)+\de)^\theta}\int\limits_{[\e s,T+\e s]}\int\limits_{B(\e y,(1+t+\e s)R)}\frac{1}{(\varrho(t-\e s,x-\e y)+\de)^\theta}\,dxdtdyds.
\end{align*}
In the last line we have used Fubini's theorem. By a change of variable, we obtain
\begin{align*}
\int\limits_{\mathcal{W}_\e}\frac{1}{\varrho_{\e,\de}^\theta}\,dxdt&\leq C\int\limits_{[0,1]}\int\limits_{\abs{y}\leq 1}\frac{1}{(\eta(s,y)+\de)^\theta}\int\limits_{[0,T]}\int\limits_{B(\textbf{0},(1+t)R)}\frac{1}{(\varrho(t,x)+\de)^\theta}\,dxdtdyds\\
&\leq C  \int\limits_{[0,1]}\int\limits_{\abs{y}\leq 1}\frac{1}{\eta^\theta(s,y)}\int\limits_{[0,T]}\int\limits_{B(\textbf{0},(1+t)R)}\frac{1}{\varrho^\theta(t,x)}\,dxdtdyds\leq C_2,
\end{align*}
where $C_2$ does not depend on $\e$ and $\de$. Note that $1/\varrho_{\e,\de}\rr1/\varrho_\e$ pointwise for $(t,x)\in\mathcal{W}_\e$ as $\de\rr0$. Hence, by the Lebesgue Dominated Convergence Theorem we have
\begin{equation*}
\int\limits_{\mathcal{W}_\e}\frac{1}{\varrho_\e^\theta}\,dxdt\leq  C_2.
\end{equation*}
}

\bigskip
\noindent\textbf{Acknowledgement.} SSG and AJ acknowledge the support of the Department of Atomic Energy, Government of India, under project no. 12-R\&D-TFR-5.01-0520. SSG would like to thank Inspire faculty-research grant DST/INSPIRE/04/2016/000237.

\end{document}